\documentclass[11pt]{amsart}

\usepackage{epigamath}

\usepackage[english]{babel}

\numberwithin{equation}{section}

\usepackage[shortlabels]{enumitem}
\setlist[enumerate,1]{label={\rm(\arabic*)}, ref={\rm\arabic*}}

\usepackage{amstext}
\usepackage{amsthm}
\usepackage{amssymb}
\usepackage[all]{xy}

\theoremstyle{plain}
\newtheorem{thm}{Theorem}[section]
\newtheorem{prop}[thm]{Proposition}
\newtheorem{lem}[thm]{Lemma}
\newtheorem{cor}[thm]{Corollary}

\theoremstyle{definition}
\newtheorem{problem}[thm]{Problem}
\newtheorem{defn}[thm]{Definition}

\theoremstyle{remark}
\newtheorem{rem}[thm]{Remark}

\renewcommand{\AA}{\mathbb{A}}
\newcommand{\FF}{\mathbb{F}}
\newcommand{\GG}{\mathbb{G}}
\newcommand{\LL}{\mathbb{L}}
\newcommand{\QQ}{\mathbb{Q}}
\newcommand{\RR}{\mathbb{R}}
\newcommand{\ZZ}{\mathbb{Z}}

\newcommand{\ba}{\mathbf{a}}
\newcommand{\bd}{\mathbf{d}}
\newcommand{\bv}{\mathbf{v}}
\newcommand{\Var}{\mathbf{Var}}

\newcommand{\cM}{\mathcal{M}}
\newcommand{\cO}{\mathcal{O}}

\newcommand{\mult}{\mathrm{mult}}
\newcommand{\st}{\mathrm{st}}
\newcommand{\sep}{\mathrm{sep}}
\newcommand{\sing}{\mathrm{sing}}

\DeclareMathOperator{\Gal}{Gal}
\DeclareMathOperator{\Hilb}{Hilb}
\DeclareMathOperator{\Image}{Im}
\DeclareMathOperator{\J}{J}
\DeclareMathOperator{\M}{M}
\DeclareMathOperator{\Spec}{Spec}
\DeclareMathOperator{\Supp}{Supp}

\renewcommand{\P}{\operatorname{P}}

\newcommand{\tpars}{(\!(t)\!)}
\newcommand{\discrep}[1]{\operatorname{discrep}\left(#1\right)}

\newcommand{\lra}{\longrightarrow}

\newcommand{\supth}[1]{\ensuremath{#1^{\mathrm{th}}}}

\EpigaVolumeYear{9}{2025} \EpigaArticleNr{28} \ReceivedOn{September 4, 2024}
\InFinalFormOn{February 20, 2025}
\AcceptedOn{July 31, 2025}

\title{Quotient singularities by permutation actions are canonical}
\titlemark{Quotient singularities by permutation actions are canonical}

\author{Takehiko Yasuda}
\address{Department of Mathematics, Graduate School of Science, the University
of Osaka, Toyonaka, Osaka 560-0043, Japan}
\address{Kavli Institute for the Physics and Mathematics of the Universe, The University of Tokyo, 5-1-5 Kashiwanoha, Kashiwa, Chiba, 277-8583, Japan}
\email{yasuda.takehiko.sci@osaka-u.ac.jp}

\authormark{T.~Yasuda}

\AbstractInEnglish{The quotient variety associated to a permutation representation of a finite group has only canonical singularities in arbitrary characteristic. Moreover, the log pair associated to such a representation is Kawamata log terminal except in characteristic two, and log canonical in arbitrary characteristic.}



\acknowledgement{The author was supported by JSPS KAKENHI Grant Numbers JP21H04994, JP23H01070 and JP24K00519.}

\begin{document}

%%%%%%%%%%%%%%%%%%%%%%%%%%%%%%%
% Title page
%%%%%%%%%%%%%%%%%%%%%%%%%%%%%%%

%\removeabove{}
%\removebetween{}
%\removebelow{}

\maketitle

\begin{prelims}

\DisplayAbstractInEnglish

\bigskip

\DisplayKeyWords

\medskip

\DisplayMSCclass

\end{prelims}

%%%%%%%%%%%%%%%%%%%%%
% Table of Contents
%%%%%%%%%%%%%%%%%%%%%

\newpage

\setcounter{tocdepth}{1}

\tableofcontents

%%%%%%%%%%%%%%%%%%%%%
% Content begins here
%%%%%%%%%%%%%%%%%%%%%

\section{Introduction}

In birational geometry, there are four important classes of singularities.
In order from the smallest class are terminal singularities, canonical
singularities, log terminal singularities and log canonical singularities.
It is well known that quotient singularities are log terminal in characteristic
zero. Moreover, we have a representation-theoretic criterion for quotient
singularities being even canonical or terminal due to Reid, Shepherd-Barron
and Tai; see \cite[Section (4.11)]{reid1987youngpersons}.
These are not generally
true in positive characteristic. It is interesting to ask the following. 

\begin{problem}
\label{prob:rep-th}Does there exist a representation-theoretic criterion
for quotient singularities belonging to these classes of singularities?
\end{problem}

Let $V$ be an $n$-dimensional affine space $\AA_{k}^{n}$ over a
field $k$ of characteristic $p\ge0$. Suppose that a finite group
$G$ acts on $V$ linearly and effectively, and let $X$ denote the
quotient variety $V/G$ and $\pi$ denote the quotient map $V\to X$.
There exists a unique effective $\QQ$-divisor $B$ on $X$ such that
$\pi^{*}(K_{X}+B)=K_{V}$. The aim of this paper is to prove the following
two theorems, giving a partial answer to the above problem. 

\begin{thm}
\label{thm:main-non-log}Suppose that $G$ acts on $V$ by permutations
of coordinates. Then, the quotient variety $X$ has only canonical
singularities. 
\end{thm}

\begin{thm}
\label{thm:main-log}Suppose that $G$ acts on $V$ by permutations
of coordinates. Then, the log pair $(X,B)$ is log canonical. Moreover,
it is Kawamata log terminal if $p\ne2$. 
\end{thm}

Note that if $p=0$ and if $G$ has no pseudo-reflections, then Theorem~\ref{thm:main-non-log} is a direct consequence of the Reid--Shepherd-Barron--Tai
criterion. Hochster and Huneke \cite[Section~6, p.~77]{hochster1992infinite}
observed that in positive characteristic, the invariant ring associated
to a permutation representation is always F-pure, and hence log canonical
from a result of Hara and Watanabe \cite{hara2002fregular}. Since
``canonical'' implies ``log canonical,'' Theorem~\ref{thm:main-non-log}
strengthens the last fact.

\begin{rem}
Hashimoto and Singh \cite[Remark 4.2]{hashimoto2023frobenius} noted
that Hochster and Huneke's observation also holds for more general
monomial representations. This may suggest that the above theorems
could be extended to monomial representations by replacing ``canonical''
in Theorem~\ref{thm:main-non-log} with ``log terminal.'' Note that
a similar approach to Theorem~\ref{thm:main-non-log} via the implication
``strongly F-regular'' $\Rightarrow$ ``log terminal'' does not
work. Firstly, ``canonical'' is stronger than ``log terminal.''
Secondly, a quotient variety $X/G$ is never strongly F-regular, provided
that $p$ divides the order of $G$ and $G$ contains no pseudo-reflections
(see \cite[Corollary~2]{broer2005thedirect} and \cite[Corollary 3.3]{yasuda2012puresubrings}).
\end{rem}

There are also a few cases  in positive characteristic where the theorems
are known. Fan \cite{fan2023crepant} constructed a crepant
resolution of the quotient variety $\AA_{k}^{4}/A_{4}$ associated
to the standard permutation action of the alternating group $A_{4}$
on $\AA_{k}^{4}$ in characteristic 2. In particular, this quotient
variety has only canonical singularities. If the symmetric group $S_{n}$
acts on $\AA_{k}^{2n}=(\AA_{k}^{2})^{n}$ by permutations of the $n$
components, then the quotient variety $\AA_{k}^{2n}/S_{n}$ is nothing
but the $\supth{n}$ symmetric product of $\AA_{k}^{2}$ and admits a crepant
resolution, the Hilbert scheme $\Hilb_{n}(\AA_{k}^{2})$ of $n$ points
on $\AA_{k}^{2}$; see \cite{beauville1983varietes,brion2005frobenius,kumar2001frobenius}.
In particular, $\AA_{k}^{2n}/S_{n}$ has only canonical singularities.
As a consequence, if a finite group $G$ acts on $\AA_{k}^{2n}$ via
a homomorphism $G\to S_{n}$, then $\AA_{k}^{2n}/G$ has only canonical
singularities.

The main ingredient of the proofs of the theorems is the wild McKay
correspondence proved in \cite{yasuda2024motivic}. It is formulated
as the following equality in a certain version of the Grothendieck
ring of varieties:
\begin{equation}
\M_{\st}(X,B)=\int_{\Delta_{G}}\LL^{n-\bv}.\label{eq:wild-McKay-1}
\end{equation}
The left-hand side is the stringy motive of the log pair $(X,B)$.
It is defined as the volume of the arc space of~$X$ with respect
to the so-called Gorenstein motivic measure with respect to the log
canonical divisor $K_{X}+B$. When the pair admits a log resolution,
the stringy motive is explicitly determined in terms of the resolution
data. In the right-hand side of the above equality, the domain $\Delta_{G}$
of the integral is a kind of moduli space of $G$-torsors over the
punctured formal disk $\Spec k\tpars$. Roughly speaking, it is the
union of countably many $k$-varieties, and if $k$ is algebraically
closed, then its $k$-points correspond to $G$-torsors over the punctured
formal disk $\Spec k\tpars$. The symbol $\LL$ means the class of
the affine line in the version of the Grothendieck ring of varieties 
that we consider and $\bv$ is a function $\Delta_{G}\to\frac{1}{\#G}\ZZ_{\ge0}$
associated to the representation $G\curvearrowright V$. For example,
when the representation is a permutation representation, which is
of our main interest in this paper, the function $\bv$ is determined
by discriminant exponents of \'etale $k\tpars$-algebras thanks to a
result in \cite{wood2015massformulas}, where $k\tpars$ denotes the
field of Laurent power series (see Section~\ref{subsec:Dimensions-of-loci}
for more details). There exists a decomposition of $\Delta_{G}$ into
countably many constructible subsets $C_{j}$, $j\in J$, such that
for each $j$, the restriction $\bv|_{C_{j}}$ of $\bv$ is constant.
The integral of the right-hand side of (\ref{eq:wild-McKay-1}) is
then defined as the countable sum $\sum_{j\in J}\{C_{j}\}\LL^{n-\bv(C_{j})}$.
If the action $G\curvearrowright V$ does not have a pseudo-reflection,
then the following variant also holds: If $X_{\sing}$ denotes the
singular locus of $X$ and $o\in\Delta_{G}$ denotes the point corresponding
to the trivial $G$-torsor, then we have
\[
\M_{\st}(X)_{X_{\sing}}=\{X_{\sing}\}+\int_{\Delta_{G}\setminus\{o\}}\LL^{n-\bv}.
\]
Here $\M_{\st}(X)_{X_{\sing}}$ is the stringy motive of $X$ along
$X_{\sing}$ and is defined as the volume of the space of arcs passing
through $X_{\sing}$. Since this invariant contains information about
the minimal discrepancy of $X$, evaluating the integral on the right-hand
side of the last equality can provide information about singularities
of $X$.

We conclude this introduction by mentioning known results related
to Problem~\ref{prob:rep-th} which have not been mentioned above.
When $p>0$ and $G$ is a cyclic $p$-group, there are representation-theoretic
criterions as desired in the problem above; see \cite{yasuda2014thepcyclic,yasuda2019discrepancies,tanno2021thewild,tanno2022onconvergence}.
Chen, Du and Gao \cite{chen2020modular} showed that 
in characteristic 3 a quotient variety by the cyclic group of order 6 containing pseudo-reflections
admits a crepant resolution and hence has only canonical singularities.
Yamamoto \cite{yamamoto2021crepant} constructed crepant resolutions
of quotient varieties $\AA_{k}^{3}/G$ in characteristic 3 for some
class of groups $G$ without pseudo-reflections, including the case
$G=S_{3}$, which shows again that the quotient variety in question
has only canonical singularities. He also showed that if $G=(\ZZ/3\ZZ)^{2}$
acts on $\AA_{k}^{3}$ without pseudo-reflections in characteristic
3, the quotient variety $\AA_{k}^{3}/G$ is not log canonical, see  \cite{yamamoto2021pathological},
while for any proper subgroup $H\subsetneq G$, the quotient variety $\AA_{k}^{3}/H$ has
only canonical singularities. This shows that the problem cannot be
reduced to the case of cyclic groups as in the case of characteristic
zero. 

The outline of the paper is as follows. In Section~\ref{sec:Preliminaries},
we collect known results which are necessary to prove our main results.
In particular, we define basic classes of singularities, then recall
basic facts on moduli spaces of $G$-torsors over $\Spec k\tpars$
as well as the wild McKay correspondence. In Section~\ref{sec:Permutation-actions},
we give key dimension estimates on loci in moduli spaces. In Section
\ref{sec:Proof of main theorems}, we prove the main theorems. 

Throughout the paper, we work over a base field $k$ of characteristic
$p\ge0$. By a variety, we mean a separated integral scheme of finite
type over $k$. 

\subsection*{Acknowledgments}
I would like to thank Ratko Darda, Linghu Fan and Mitsuyasu Hashimoto
for inspiring conversations and helpful comments.

\section{Preliminaries\label{sec:Preliminaries}}

\subsection{Singularities in the minimal model program}

Let $(X,B)$ be a log pair. Namely, $X$ is a normal variety over
$k$ and $B$ is a $\QQ$-divisor on $X$ such that $K_{X}+B$ is
$\QQ$-Cartier. For a proper birational morphism $f\colon Y\to X$
with $Y$ normal and a prime divisor $E$ on $Y$, the discrepancy
of $(X,B)$ at $E$, denoted by $\discrep{E;X,B}$, is defined to be
the multiplicity of $K_{Y}-f^{*}(K_{X}+B)$. 

\begin{defn}
We say that $(X,B)$ is \emph{Kawamata log terminal} (resp.~\emph{log
canonical}\,) if $\discrep{E;X,B}>-1$ (resp.~$\discrep{E;X,B}\ge-1$)
for every proper birational morphism $f\colon Y\to X$ and every prime
divisor $E$ on~$Y$. 
\end{defn}

We identify a $\QQ$-Gorenstein normal variety $X$ with the log pair
$(X,0)$. 

\begin{defn}
We say that a $\QQ$-Gorenstein normal variety $X$ has only \emph{terminal
singularities }(resp.~\emph{canonical} \emph{singularities}) if $\discrep{E;X}:=\discrep{E;X,0}>0$
(resp.~$\discrep{E;X}\ge0$) for every proper birational morphism
$f\colon Y\to X$ with $Y$ a normal variety and every prime divisor
$E$ on $Y$.
\end{defn}

\subsection{Stringy motives\label{subsec:Stringy-motives}}

The Grothendieck ring of varieties, denoted by $K_{0}(\Var_{k})$,
is defined as the quotient of the free abelian group generated by
the isomorphism classes of $k$-varieties by the so-called scissor
relation. It has a natural structure of commutative ring. We denote
the class of a $k$-variety $Z$ in this ring or in its variants by
$\{Z\}$. The class $\{\AA_{k}^{1}\}$ of the affine line $\AA_{k}^{1}$
plays a special role and is denoted by $\LL$. We need the version
of the Grothendieck ring of varieties that was denoted by $\widehat{\cM'_{k,r}}$
in \cite{yasuda2024motivic}. Here $r$ denotes a positive integer
which is factorial enough so that all the $\QQ$-Cartier divisors
that we will consider are $r$-Cartier and finite groups that we will
consider have orders dividing $r$. To obtain this ring, we modify
$K_{0}(\Var_{k})$ by adjoining the fractional power $\LL^{1/r}$
of $\LL$ and its inverse, taking a quotient by imposing an extra
relation among elements, and taking the completion with respect to
some filtration. In particular, the ring $\widehat{\cM'_{k,r}}$ contains
fractional powers $\LL^{n/r}$, $n\in\ZZ$, of $\LL$. Since the ring
is also complete with respect to a certain topology, we can consider
some type of infinite sums in it and discuss their convergence and
divergence. 

Let $(X,B)$ be a log pair. We can define a motivic measure on the arc
space $\J_{\infty}X$ of $X$ in terms of the log canonical divisor
$K_{X}+B$, which is often called the Gorenstein (or $\QQ$-Gorenstein)
measure. The \emph{stringy motive} $\M_{\st}(X,B)$ is defined to be
the volume of the entire arc space $\J_{\infty}X$ with respect to this
measure. It is expressed as a countable sum
$\sum_{i}\{Z_{i}\}\LL^{m_{i}}$, where the $Z_{i}$ are $k$-varieties
and the $m_{i}$ are elements of $\frac{1}{r}\ZZ$.  This sum converges
if and only if for each $s\in\RR$, there are at most finitely many
indices $i$ such that $\dim Z_{i}+m_{i}\ge s$.  If this is the case,
$\M_{\st}(X,B)$ is an element of the ring $\widehat{\cM'_{k,r}}$;
otherwise, we formally put it to be $\infty$. In either case, we define
the\emph{ dimension} of $\M_{\st}(X,B)$ as
\[
\dim\M_{\st}(X,B):=\sup_{i}(\dim Z_{i}+m_{i})\in\frac{1}{r}\ZZ\cup\{\infty\}.
\]
When $B=0$, we write $\M_{\st}(X,0)$ simply as $\M_{\st}(X)$. The
\emph{stringy motive of $X$ along the singular locus} $X_{\sing}$,
denoted by $\M_{\st}(X)_{X_{\sing}}$, is defined to be the volume
of the space of those arcs on $X$ that pass through $X_{\sing}$.
Its dimension is similarly defined.

When $X$ is smooth, we have $\M_{\st}(X)=\{X\}$. In the general
case, $\M_{\st}(X,B)$ can be regarded as a modification of $\{X\}$
obtained by incorporating singularities of the pair $(X,B)$. Suppose
that there exists a log resolution $f\colon Y\to X$ of the pair and
write 
\[
K_{Y}-f^{*}(K_{X}+B)=\sum_{i\in I}a_{i}E_{i},
\]
where the $E_{i}$ are prime divisors on $Y$ and the $a_{i}$ are rational
numbers. Then, we can express $\M_{\st}(X,B)$ by the following formula: 
\[
\M_{\st}(X,B)=\begin{cases}
\sum_{I'\subset I}\{E_{I'}^{\circ}\}\prod_{i\in I'}\frac{\LL-1}{\LL^{a_{i}+1}-1} & (\forall i,\,a_{i}>-1),\\
\infty & (\text{otherwise}).
\end{cases}
\]
In particular, assuming the existence of a log resolution, we have
that $\M_{\st}(X,B)\ne\infty$ if and only if $(X,B)$ is Kawamata
log terminal. Note however that stringy motives are defined whether
the pair admits a log resolution or not. When $B=0$, we have the
following equality: 
\[
\inf_{(f,E);\,f(E)\subset X}\discrep{E;X}=d-1-\dim\M_{\st}(X)_{X_{\sing}}\in\frac{1}{r}\ZZ\cup\{-\infty\}.
\]
Here the pair $(f,E)$ runs over the pairs of a proper birational
morphism $f\colon Y\to X$ with $Y$ a normal variety and 
a prime divisor $E$ on $Y$ such that $f(E)\subset X_{\sing}$. In particular,
$X$ has only canonical singularities if and only if 
\[
\dim\M_{\st}(X)_{X_{\sing}}\le d-1.
\]
For more details about the relation between stringy motives and discrepancies,
we refer the reader to \cite[Section 2]{yasuda2019discrepancies},
\cite[Section 16]{yasuda2024motivic} and \cite[Section 6.6]{yasuda2021motivic}.

\subsection{P-Moduli spaces of formal torsors}

In \cite{tonini2023moduliof}, the authors developed the theory of
P-schemes and P-moduli spaces in order to construct moduli spaces
of torsors over the punctured formal disk $\Spec k\tpars$. A \emph{P-morphism}
$f\colon Y\to X$ of schemes over $k$ is a collection of compatible
maps $f(L)\colon Y(L)\to X(L)$ for algebraically closed fields $L/k$
such that there exist a surjective morphism $Z\to Y$ locally of finite
presentation (a sur covering) and a morphism $Z\to X$ such that for
each algebraically closed field $L/K$, the following diagram of the
induced maps is commutative:
\[
\xymatrix{Z(L)\ar[d]\ar[dr]\\
Y(L)\ar[r] & X(L)\rlap{.}
}
\]
The \emph{category of P-schemes} \emph{over} $k$ has schemes over
$k$ as its objects and P-morphisms as its morphisms. An important
property of this category is the following: When $X$ and $Y$ are
locally of finite type and separated over $k$,  a P-morphism
$f\colon Y\to X$ is an isomorphism in the category of P-schemes if
and only if for every algebraically closed field $L/k$, $f(L)$ is
bijective; see \cite[Lemma 4.32]{tonini2023moduliof}. Since an ordinary
morphism of $k$-schemes induces a P-morphism in the obvious way,
there exists a natural functor from the category of schemes over $k$
to the one of P-schemes over $k$: 
\begin{align*}
\P\colon\{\text{scheme over \ensuremath{k}}\} & \lra\{\text{P-scheme over \ensuremath{k}}\}\\
X & \longmapsto X^{\P}.
\end{align*}

Fixing a finite group $G$, let us consider the functor 
\begin{align*}
F_{G}\colon\{\text{affine scheme over \ensuremath{k}}\} & \lra\{\text{Set}\}\\
\Spec R & \longmapsto\{G\text{-torsor over \ensuremath{\Spec R\tpars}}\}/{\cong}.
\end{align*}
Here $R\tpars$ denotes the ring of Laurent power series with coefficients
in $R$. This functor has a strong P-moduli space which is of the
form $(\coprod_{i\in I}W_{i})^{\P}$, where $I$ is a countable set
and the  $W_{i}$ are $k$-varieties. Roughly speaking, this means that
for each $G$-torsor over $\Spec R\tpars$, we have the induced P-morphism
$\Spec R\to\coprod_{i\in I}W_{i}$, and for each algebraically closed
field $L/k$, the induced map
\[
\{G\text{-torsor over \ensuremath{\Spec L\tpars}}\}/{\cong}\lra\left(\coprod_{i\in I}W_{i}\right)(L)
\]
is bijective. See \cite[Definition 4.16]{tonini2023moduliof} for
the precise meaning. We denote this P-moduli space by $\Delta_{G}$.
The P-scheme $(\coprod_{i\in I}W_{i})^{\P}$ is unique up to unique
isomorphism; if $\coprod_{j\in J}V_{j}$ is another scheme satisfying
the same property, then there exist a scheme $Z$ and morphisms $Z\to\coprod_{i\in I}W_{i}$
and $Z\to\coprod_{j\in J}V_{j}$ which are of finite type and geometrically
bijective; see \cite[Corollary 4.33]{tonini2023moduliof}. In what follows,
we choose one scheme-model $\coprod_{i\in I}W_{i}$ of the P-scheme
$\Delta_{G}$ and identify $\Delta_{G}$ with it. 

\subsection{The wild McKay correspondence}

Let $V=\AA_{k}^{n}$ and let $G$ be a finite group which acts on
$\AA_{k}^{n}$ linearly and effectively. Let $X=V/G$ be the associated
quotient variety. There exists a unique effective $\QQ$-divisor on $B$
which has support along the branch divisor of the quotient map $V\to
X$ and satisfies $\pi^{*}(K_{V/G}+B)=K_{V}$. Note that $B=0$ if and
only if $X\to V$ is \'etale in codimension one, which holds if and
only if $G$ contains no pseudo-reflection. (An element $g\in G$ is
called a \emph{pseudo-reflection }if the fixed-point locus $V^{G}$ has
codimension one in $V$.)

Let $\Delta_{G}=\coprod_{i\in I}W_{i}$ be as above. We call a subset
$S\subset\Delta_{G}$ \emph{constructible} if $S=\bigcup_{j\in J}S_{j}$
for a finite subset $J\subset I$ and constructible subsets $S_{j}\subset W_{j}$,
$j\in J$. To the given linear action of $G$ on $V$, we can associate
a function $\bv\colon\Delta_{G}\to\frac{1}{\#G}\ZZ_{\ge0}$, for example,
following \cite[Definition 3.3]{wood2015massformulas} or \cite[Definition 9.1]{tonini2023moduliof}.
This function has the following properties:
\begin{enumerate}
\item If $o$ denotes the point of $\Delta_{G}$ corresponding to the trivial
$G$-torsor, then $\bv(o)=0$. The function $\bv$ takes positive
values on $\Delta_{G}\setminus\{o\}$. 
\item There exist countably many constructible subsets $C_{j}$, $j\in J$, 
of $\Delta_{G}\setminus\{o\}$ such that $\Delta_{G}\setminus\{o\}=\bigsqcup_{j\in J}C_{j}$
and $\bv|_{C_{j}}$ is constant for each $j$.
\end{enumerate}
With this notation, we can define the following motivic integrals:
\begin{align*}
\int_{\Delta_{G}}\LL^{-\bv} & :=1+\sum_{j\in J}\{C_{j}\}\LL^{-\bv(C_{j})},\\
\int_{\Delta_{G}\setminus\{o\}}\LL^{-\bv} & :=\sum_{j\in J}\{C_{j}\}\LL^{-\bv(C_{j})}.
\end{align*}
If the infinite sum on the right-hand side converges in the sense
explained in Section~\ref{subsec:Stringy-motives}, then the integrals
$\int_{\Delta_{G}}\LL^{-\bv}$ and $\int_{\Delta_{G}\setminus\{o\}}\LL^{-\bv}$
are defined as elements of the ring $\widehat{\cM'_{k,r}}$ in \cite{yasuda2024motivic}.
Otherwise, we put 
\[
\int_{\Delta_{G}}\LL^{-\bv}=\int_{\Delta_{G}\setminus\{o\}}\LL^{-\bv}=\infty.
\]
Whether this condition holds or not, the dimensions of these integrals
are well defined. 

{\samepage
\begin{thm}[\textit{cf.} {\cite[Corollary 1.4 and the proof of Corollary 16.4]{yasuda2024motivic}}]\leavevmode
\begin{enumerate}
\item We have 
\[
\M_{\st}(X,B)=\int_{\Delta_{G}}\LL^{n-\bv}.
\]
\item If $G$ has no pseudo-reflection, then 
\[
\M_{\st}(X)_{X_{\sing}}=\{X_{\sing}\}+\int_{\Delta_{G}\setminus\{o\}}\LL^{n-\bv}.
\]
\end{enumerate}
\end{thm}
}

\begin{prop}[\textit{cf.} {\cite[Corollary  1.4]{yasuda2024motivic}}]
\label{prop:canonical-criterion}Suppose that $G$ contains no pseudo-reflection.
Fix a decomposition $\Delta_{G}\setminus\{o\}=\bigsqcup_{j\in J}C_{j}$
as above. If $\dim C_{j}-\bv(C_{j})\le-1$ for every $j\in J$, then
$X$ has only canonical singularities.
\end{prop}

\begin{proof}
From \cite[Theorem 1.2]{yasuda2019discrepancies}, $X$ has only canonical
singularities if and only if 
\[
n-1-\max\left\{ \dim X_{\sing},\max_{j}\left(n+\dim C_{j}-\bv\left(C_{j}\right)\right)\right\} \ge0.
\]
This inequality is equivalent to that for every $j$, we have $\dim C_{j}-\bv(C_{j})\le-1$.
\end{proof}

\begin{prop}
\label{prop:klt-lc}We keep the notation as above.
\begin{enumerate}
\item\label{p:klt-lc-1} If 
\[
\lim_{j\in J}\dim C_{j}-\bv(C_{j})=-\infty,
\]
then the pair $(X,B)$ is Kawamata log terminal. 
\item\label{p:klt-lc-2} If\, $\sup\{\dim C_{j}-\bv(C_{j})\mid j\in J\}<+\infty$, then the pair $(X,B)$ is log canonical.
\end{enumerate}
\end{prop}

\begin{proof}
Note that when $G$ has no pseudo-reflections,
assertion~\eqref{p:klt-lc-1} is the same as \cite[Corollary
  1.4(2)]{yasuda2024motivic}. Let~$Y$ be a normal variety, and let
$Y\to X$ be a proper birational morphism.  Let $B'$ be the
$\QQ$-divisor on~$Y$ such that $(Y,B')$ is crepant over $(X,B)$. From
\cite[Theorem 16.2]{yasuda2024motivic},
$\M_{\st}(Y,B')=\M_{\st}(X,B)$.  If $(X,B)$ is not Kawamata log
terminal (resp.~log canonical), then for some proper birational
morphism $Y\to X$, the divisor~$B'$ has multiplicity at most $-1$ at some
prime divisor $E$. Then the standard computation of stringy invariants
(in a neighborhood of a general point of $E$) shows that
$\M_{\st}(Y,B')$ does not converge. Translating these conditions to
ones on $\dim C_{j}-\bv(C_{j})$ shows assertion~\eqref{p:klt-lc-1}. If
$(X,B)$ was not log canonical, a similar reasoning shows that the 
dimensions of terms in an infinite sum defining
$\M_{\st}(Y,B')=\M_{\st}(X,B)=\int_{\Delta_{G}}\LL^{n-\bv}$ are not
bounded above, which shows assertion~\eqref{p:klt-lc-2}.
\end{proof}

\section{Permutation actions and dimension estimates\label{sec:Permutation-actions}}

From now on, suppose that $G$ is a subgroup of $S_{n}$ and acts
on $V=\AA_{k}^{n}$ by the induced permutation action. Let $X=V/G$
be the quotient variety with quotient map $\pi\colon V\to X$, and
let $B$ be the effective $\QQ$-divisor on~$X$ such that $\pi^{*}(K_{X}+B)=K_{V}$. 

\subsection{The Gorenstein index of $\boldsymbol{X}$}

\begin{lem}
\label{lem:1-2-Gor}The quotient variety $X$ is $2$-Gorenstein; 
that is, $2K_{X}$ is Cartier. Moreover, if $p=2$, then $X$ is $1$-Gorenstein,
that is, $K_{X}$ is Cartier.
\end{lem}

\begin{proof}
We first consider the case $p\ne2$. Let $x_{1},\dots,x_{n}$ be coordinates
of the affine space $V$. Let $R$ be the reduced divisor on $V$
defined by 
\begin{equation}
\prod_{\substack{i>j\\
(i,j)\in G
}
}(x_{i}-x_{j})=0.\label{eq:diff-prod}
\end{equation}
Here $(i,j)$ denotes a transposition. We easily see that the ramification
divisor of $\pi\colon V\to X$ has support $R$. The ramification
index of $\pi$ along every irreducible component of $R$ is 2. In
particular, $\pi$ is tamely ramified at general points of $R$. It
follows that $(\pi^{*}\omega_{X})^{**}\cong\omega_{V}(-R)$ (for example,
see \cite[Section 2.41]{kollar2013singularities}). The sheaf $\omega_{V}(-R)$
is generated by the global section
\begin{equation}
\eta:=\prod_{\substack{i>j\\
(i,j)\in G
}
}(x_{i}-x_{j})\cdot dx_{1}\wedge\cdots\wedge dx_{n}.\label{eq:eta}
\end{equation}
For $\sigma\in G$, we have
\begin{align*}
\sigma(\eta) & =\prod_{\substack{i>j\\
(i,j)\in G
}
}(x_{\sigma(i)}-x_{\sigma(j)})\cdot dx_{\sigma(1)}\wedge\cdots\wedge dx_{\sigma(n)}\\
 & =\pm\prod_{\substack{i>j\\
(i,j)\in G
}
}(x_{\sigma(i)}-x_{\sigma(j)})\cdot dx_{1}\wedge\cdots\wedge dx_{n}.
\end{align*}
Since $(\sigma(i),\sigma(j))=\sigma\circ(i,j)\circ\sigma^{-1}$, the
transposition $(\sigma(i),\sigma(j))$ runs over all transpositions
in $G$ when $(i,j)$ does so. Therefore, $\sigma(\eta)=\pm\eta$.
It follows that $\eta^{\otimes2}\in(\omega_{V}(-R))^{\otimes2}$ is
invariant under the $G$-action and gives a global generator of $\omega_{X}^{[2]}=(\omega_{X}^{\otimes2})^{**}$.
This shows $X$ is 2-Gorenstein.

We next consider the case $p=2$. The ramification divisor of
the quotient map $\pi\colon V\to X$ is again the reduced divisor
$R$ defined by the same equation as in the case $p\ne2$. To understand
the ramification of~$\pi$ along $R$, we first consider the case
$G=\{1_{G},(1,2)\}$. The coordinate ring of $X$ is then 
\[
k[x_{1}+x_{2},x_{1}x_{2},x_{3},\dots,x_{n}]\subset k[x_{1},x_{2},x_{3},\dots,x_{n}].
\]
In particular, $X\cong\AA_{k}^{n}$. Pulling back the generator 
\[
d(x_{1}+x_{2})\wedge d(x_{1}x_{2})\wedge dx_{3}\wedge\cdots\wedge dx_{n}
\]
of $\omega_{X}$ by $\pi$, we get the $n$-form
\begin{align*}
 & (dx_{1}+dx_{2})\wedge(x_{2}dx_{1}+x_{1}dx_{2})\wedge dx_{3}\wedge\cdots\wedge dx_{n}\\
 & =(x_{1}+x_{2})dx_{1}\wedge dx_{2}\wedge dx_{3}\wedge\cdots\wedge dx_{n}.
\end{align*}
This is a generator of $\omega_{V}(-R)$, where $R$ is defined by
the function $x_{1}+x_{2}$ in the current situation. Thus, we get
$\pi^{*}\omega_{X}=\omega_{V}(-R)$. Let us return to the case of
a general group $G$ and take a general $\overline{k}$-point $b=(b_{1},\dots,b_{n})$
of $R$ with $b_{i}=b_{j}$ for some $(i,j)$. Here $\overline{k}$
denotes an algebraic closure of $k$. The stabilizer group of this
point is generated by the transposition $(i,j)$, which interchanges
$x_{i}-b_{i}$ and $x_{j}-b_{j}$ for the local coordinates $x_{1}-b_{1},\dots,x_{n}-b_{n}$.
This shows that the morphism 
\[
\Spec\widehat{\cO_{V_{\overline{k}},b}}\lra\Spec\widehat{\cO_{X_{\overline{k}},\overline{b}}}
\]
between the formal neighborhoods of $b$ and of its image $\overline{b}$
on $X$ is isomorphic to the one in the case where $G=\{1_{G},(1,2)\}$
and $b$ is the origin. Thus, the ramification of $\pi\colon V\to X$
looks like the above special case, and we still have $(\pi^{*}\omega_{X})^{**}=\omega_{V}(-R)$
in the general case. The $n$-form $\eta$ given by Formula (\ref{eq:eta})
is again a global generator of $\omega_{V}(-R)$. Since we are working
in characteristic 2, the same computation as in the case $p\ne2$
shows 
\[
\sigma(\eta)=\pm\eta=\eta.
\]
Thus, $\eta$ is $G$-invariant and gives a global generator of $\omega_{X}$,
which shows that $X$ is 1-Gorenstein. 
\end{proof}

\begin{lem}
\label{lem:coef-boundary}If $p\ne2$, then the multiplicity of\, $B$
at every irreducible component of its support is $1/2$. If $p=2$,
the multiplicity of\, $B$ at every irreducible component of its support
is $1$.
\end{lem}

\begin{proof}
If $p\ne2$, then $\pi\colon V\to X$ is tamely ramified along general
points of the ramification divisor $R$ with ramification index 2.
As is well known, the multiplicity of the branch divisor at a prime
divisor is then given by $1/2$. Suppose that $p=2$. As in the proof
of Lemma~\ref{lem:1-2-Gor}, we may reduce to the case $G=\{1_{G},(1,2)\}$.
Let $B'=\pi(R)=\Supp(B)$. The prime divisor $B'$ is defined by the
ideal $x_{1}+x_{2}$ of the coordinate ring
\[
k[x_{1}+x_{2},x_{1}x_{2},x_{3},\dots,x_{n}].
\]
Thus, $\pi^{*}B'=R$. In the proof of Lemma~\ref{lem:1-2-Gor}, we
showed $\pi^{*}\omega_{X}=\omega_{V}(-R)$. It follows that $\pi^{*}(\omega_{X}(B'))=\omega_{V}$
and hence $B=B'$, which shows the lemma in the case $p=2$.
\end{proof}

\begin{cor}
\label{cor:Cartier}The $\QQ$-divisor $B$ is 2-Cartier. Moreover,
it is Cartier if $p=2$.
\end{cor}

\begin{proof}
We first consider the case $p\ne2$. Consider the log pair $(V/S_{n},D)$
associated to the standard permutation action of $S_{n}$. As is well known,
the support of $D$ is defined by the discriminant polynomial, and
the coefficient of $D$ is $1/2$ from Lemma~\ref{lem:coef-boundary}.
Since $V/S_{n}$ is isomorphic to $\AA_{k}^{n}$, so, in particular, smooth,
$K_{V/S_{n}}+D$ is 2-Cartier. If $h\colon X\to V/S_{n}$ denotes
the natural map, then 
\[
K_{X}+B=h^{*}\left(K_{V/S_{n}}+D\right).
\]
 Since $K_{X}$ and $h^{*}(K_{V/S_{n}}+D)$ are both 2-Cartier (see Lemma
(\ref{lem:1-2-Gor})), $B$ is also 2-Cartier. 

We next consider the case $p=2$ and again consider the log pair $(V/S_{n},D)$
associated to the standard action of $S_{n}$ on $V$. In this case,
$D$ has coefficient 1 from Lemma~\ref{lem:coef-boundary}. Hence
$K_{V/S_{n}}+D$ is 1-Cartier, and so is 
\[
K_{X}+B=h^{*}\left(K_{V/S_{n}}+D\right).
\]
Since $K_{X}$ and $K_{V/S_{n}}+D$ are both Cartier (see Lemma~\ref{lem:1-2-Gor}),
$B$ is also Cartier. 
\end{proof}

\subsection{Dimensions of loci in moduli spaces\label{subsec:Dimensions-of-loci}}

The inclusion map $G\hookrightarrow S_{n}$ induces a map $\Delta_{G}\to\Delta_{S_{n}}$
which sends a $G$-torsor $P$ to the contracted product $P\wedge^{G}S_{n}$
(for example, see \cite[Proposition~2.2.2.12]{calm`es2015groupes} or \cite[Section 4.4, Item 5]{emsalem2017twisting}). (Strictly speaking, we have a canonical P-morphism $\Delta_{G}\to\Delta_{S_{n}}$.
To find a scheme-morphism inducing this P-morphism, we need to replace
the chosen scheme-model $\coprod_{j\in J}W_{j}$ of $\Delta_{G}$
with another model $\coprod_{i\in I}V_{i}$ given with a universal
bijection $\coprod_{i\in I}V_{i}\to\coprod_{j\in J}W_{j}$.) The map
$\Delta_{G}\to\Delta_{S_{n}}$ is quasi-finite, and the cardinality
of each fiber is at most $\#S_{n}/\#G$. Let $\Delta_{n}$ denote
the P-moduli space of degree $n$ finite \'etale covers of $\Spec k\tpars$.
We have an isomorphism $\Delta_{S_{n}}\to\Delta_{n}$, sending an
$S_{n}$-torsor $A\to\Spec k\tpars$ to $A/S_{n-1}\to\Spec k\tpars$,
where $S_{n-1}$ is identified with the stabilizer of $1\in\{1,2,\dots,n\}$
(see \cite[Proposition 2.7]{tonini2020essentially} and \cite[Definition 8.2]{tonini2023moduliof}).
We denote by $\psi_{G}$ the composite map 
\[
\psi_{G}\colon\Delta_{G}\lra\Delta_{S_{n}}\lra\Delta_{n}.
\]
Since this is quasi-finite, for every constructible subset $C\subset\Delta_{G}$,
we have $\dim C=\dim\psi_{G}(C)$. 

Let $\bd\colon\Delta_{n}\to\ZZ_{\ge0}$ be the discriminant exponent
function. From \cite[Lemma 2.6 and Theorem 4.8]{wood2015massformulas},
the $\bv$-function on $\Delta_{G}$ (associated to the permutation
action $G\curvearrowright\AA_{k}^{n}$ defined as above) factors as
follows:
\[
\bv\colon\Delta_{G}\xrightarrow{\psi_{G}}\Delta_{n}\xrightarrow{\bd/2}\frac{1}{2}\ZZ_{\ge0}.
\]
From \cite[Theorem 9.8]{tonini2023moduliof} (see also \cite[Lemma 14.3]{yasuda2024motivic}),
the function $\bv$ is locally constructible. Namely, there
exists a decomposition of $\Delta_{G}$ into countably many constructible
subsets $C_{i}$, $i\in I$, such that $\bv|_{C_{j}}$ is constant
for every $i$. Since $\bd/2$ is identical to the $\bv$-function associated to
the standard permutation action of~$S_{n}$ on $\AA_{k}^{n}$, the
function $\bd/2$ on $\Delta_{n}$, as well as $\bd$, is locally constructible.
This shows that for each $d\in\ZZ_{\ge0}$, 
\[
\Delta_{n,d}:=\bd^{-1}(d)
\]
is a locally constructible subset of $\Delta_{n}$. In fact, this
set is also quasi-compact and constructible; see \cite[Corollary 4.12]{yasuda2024motivic2}.

Let $\Delta_{n}^{\circ}\subset\Delta_{n}$ denote the subspace of
geometrically connected covers (see \cite[Definition 8.6]{tonini2023moduliof}),
which is a locally constructible subset of $\Delta_{n}$; see \cite[Lemma 8.7]{tonini2023moduliof}.
Let $\nu=(\nu_{1},\nu_{2},\dots,\nu_{l})$ be a partition of $n$
by positive integers; that is, the $\nu_{i}$ are positive integers satisfying
$n=\sum_{i=1}^{l}\nu_{i}$. We have the map
\[
\eta_{\nu}\colon\prod_{i=1}^{l}\Delta_{\nu_{i}}^{\circ}\lra\Delta_{n},\quad(A_{i})_{1\le i\le l}\longmapsto\coprod_{i=1}^{l}A_{i}.
\]
For $d\ge0$, let 
\[
\Delta_{n,d}^{\circ}:=\Delta_{n}^{\circ}\cap\Delta_{n,d}.
\]
If $\delta=(\delta_{1},\dots,\delta_{l})$ is a partition of $d$
by non-negative integers, then $\eta_{\nu}$ restricts to 
\[
\eta_{\nu,\delta}\colon\prod_{i}^{l}\Delta_{\nu_{i},\delta_{i}}^{\circ}\lra\Delta_{n,d}.
\]
This map is quasi-finite; in particular, the source of the map has
the same dimension as its image. Moreover, $\Delta_{n,d}$ is covered
by the images of the maps $\eta_{\nu,\delta}$ as $\nu$ and $\delta$
run over partitions as above. Thus, we can estimate the dimension
of $\Delta_{n,d}$ in terms of those  of the $\Delta_{\nu_{i},\delta_{i}}^{\circ}$.

The following result from \cite{yasuda2024motivic2} is the key in
the proofs of our main results. 

\begin{thm}[\textit{cf.} \cite{yasuda2024motivic2}]
\label{thm:dimensions-loci} We have
\[
\dim\Delta_{n,d}^{\circ}=\begin{cases}
0 & (p\nmid n,\,d=n-1),\\
-\infty & (p\nmid n,\,d\ne n-1),\\
\lceil(d-n+1)/p\rceil & (p\mid n,\,p\nmid(d-n+1),\,d-n+1\ge0),\\
-\infty & (p\mid n,\,p\nmid(d-n+1),\,d-n+1<0),\\
-\infty & (p\mid n,\,p\mid(d-n+1)).
\end{cases}
\]
Here we follow the convention that $\dim\emptyset=-\infty$ and that
if $p=0$, then $p\nmid n$. 
\end{thm}

\begin{rem}
Theorem~\ref{thm:dimensions-loci} can be regarded as a motivic version of a formula
by Krasner \cite{krasner1966nombredes,krasner1962nombredes}. See
\cite{yasuda2024motivic2} for more details. 
\end{rem}

\begin{cor}
\label{cor:loci-dim-bound}Let $d$ be a positive integer, and let
$n$ be an integer with $n\ge2$.
% query 
\begin{enumerate}
\item\label{c:ldb-1} If $n\ge4$, then 
\[
\dim\Delta_{n,d}^{\circ}-\frac{d}{2}\le-1.
\]
\item\label{c:ldb-2} If $n=3$, then
\[
\dim\Delta_{n,d}^{\circ}-\frac{d}{2}\le\begin{cases}
-1 & (p\ne3),\\
-{1}/{2} & (p=3).
\end{cases}
\]
\item\label{c:ldb-3} If $n=2$, then
\[
\dim\Delta_{n,d}^{\circ}-\frac{d}{2}=\begin{cases}
-1/2 & (p\ne2,\,d=1),\\
-\infty & (p\ne2,\,d>1),\\
0 & (p=2,\,d\text{ even}),\\
-\infty & (p=2,\,d\text{ odd}).
\end{cases}
\]
\item\label{c:ldb-4} We always have
\[
\dim\Delta_{n,d}^{\circ}-\frac{d}{2}\le0.
\]
\end{enumerate}
\end{cor}

\begin{proof}
If $\Delta_{n,d}^{\circ}$ is empty, then it has dimension $-\infty$
and the desired inequalities are obvious. In what follows, we only
consider the situation where $\Delta_{n,d}^{^{\circ}}$ is not empty. 

\eqref{c:ldb-1}~  If $p\nmid n$ and $d=n-1$, then 
\[
\dim\Delta_{n,d}^{\circ}-\frac{d}{2}=0-\frac{n-1}{2}\le-\frac{3}{2}\le-1.
\]
If $p\mid n$, $p\nmid(d-n+1)$ and $d-n+1\ge0$, then 
\[
\dim\Delta_{n,d}^{\circ}-\frac{d}{2}=\left\lceil \frac{d-n+1}{p}\right\rceil -\frac{d}{2}.
\]
Let $h(d)$ denote the right-hand side. If $d-n+1$ and $(d+1)-n+1$
are both coprime to $p$, then 
\[
h(d+1)=h(d)-\frac{1}{2}<h(d).
\]
If $(d+1)-n+1$ is divisible by $p$, then  
\[
h(d+2)=h(d)+1-1=h(d).
\]
Thus, $h(d)$ is a weakly decreasing function, as $d$ runs over integers
satisfying $p\nmid(d-n+1)$ and $d-n+1\ge0$, and attains the maximum
when $d-n+1=1$, equivalently when $d=n$. Thus, the maximum of $h(d)$
is
\[
\left\lceil \frac{d-n+1}{p}\right\rceil -\frac{d}{2}=1-\frac{n}{2}\le-1.
\]
Here the last inequality follows from the assumption $n\ge4$. 

\eqref{c:ldb-2}~  If $p\ne3$ and $d=2$, then 
\[
\dim\Delta_{n,d}^{\circ}-\frac{d}{2}=-1.
\]
Suppose $p=3$. If $3\nmid(d-2)$ and $d-2\ge0$, then
\begin{align*}
\dim\Delta_{3,d}^{\circ}-\frac{d}{2} & \le\left\lceil \frac{d-2}{3}\right\rceil -\frac{d}{2}\\
 & \le\frac{d-2}{3}+\frac{2}{3}-\frac{d}{2}\\
 & \le\frac{d}{3}-\frac{d}{2}\\
 & <0.
\end{align*}
Since $\dim\Delta_{3,d}^{\circ}-{d}/{2}$ is a half integer, it
is at most $-1/2$.

\eqref{c:ldb-3}~  The only nontrivial case is when $p=2$ and $d$ is even. In this
case, 
\[
\dim\Delta_{n,d}^{\circ}=\left\lceil \frac{d-1}{2}\right\rceil =\frac{d}{2}.
\]

\eqref{c:ldb-4}~  This is a direct consequence of the previous three assertions.
\end{proof}

\subsection{Better estimations derived from the non-existence of transposition\label{subsec:Non-existence-of-transposition}}

Although Corollary~\ref{cor:loci-dim-bound} is good enough for most
cases, it is not enough to show Theorem~\ref{thm:main-non-log} in
a few exceptional cases. Assuming that $G$ contains no pseudo-reflections,
we obtain better estimations in characteristics 2 and 3. 

In what follows, we denote by $K$ the field $k\tpars$ of Laurent
power series. Let $\Gamma_{K}:=\Gal(K^{\sep}/K)$ be its absolute
Galois group. Let $\wp\colon K\to K$ be the Artin--Schreier map
$f\mapsto f^{p}-f$. 

\begin{lem}
\label{lem:indep-trans}Suppose $p=2$. Let $L_{1},\dots,L_{m}$ be
quadratic field extensions of $K$ such that the corresponding elements
$[L_{1}],\dots,[L_{m}]\in K/\wp K$ $($via Artin--Schreier theory$)$
are linearly independent over $\FF_{2}$. Let $\varphi\colon\Gamma_{K}\to S_{n}$
be the map corresponding to the \'etale $K$-algebra 
\[
L=L_{1}\times\cdots\times L_{m}\times K^{n-2m}.
\]
Then, $\Image(\varphi)$ contains a transposition. 
\end{lem}

\begin{proof}
Let $\alpha_{1},\dots,\alpha_{m}\colon\Gamma_{K}\to\FF_{2}$ be the
maps corresponding to $L_{1},\dots,L_{m}$, respectively. After a
suitable rearrangement of elements of $\{1,\dots,n\}$, $\Image(\varphi)$
is contained in 
\[
\FF_{2}^{m}=\langle(1,2),(3,4),\dots,(2m-1,2m)\rangle\subset S_{n}.
\]
The induced morphism $\Gamma_{K}\to\FF_{2}^{m}$ is given by
\[
\Gamma_{K}\ni\gamma\longmapsto(\alpha_{1}(\gamma),\dots,\alpha_{m}(\gamma))\in\FF_{2}^{m}, 
\]
and its image is an $\FF_{2}$-linear subspace from Artin-Schreier
theory. Since $\alpha_{1},\dots,\alpha_{m}$ are linearly independent
over $\FF_{2}$, there is no non-zero linear map $\FF_{2}^{m}\to\FF_{2}$
such that the composition $\Gamma_{K}\to\FF_{2}^{m}\to\FF_{2}$ is
the trivial map onto $0\in\FF_{2}$. This is equivalent to that $\Gamma_{K}\to\FF_{2}^{m}$
is surjective. In particular, the image of $\Gamma_{K}\to G\subset S_{n}$
contains the transposition $(1,2)\in\langle(1,2),(3,4),\dots,(2m-1,2m)\rangle=\FF_{2}^{m}$. 
\end{proof}

\begin{lem}
\label{lem:non-Gal-trans}Suppose $p=3$. Let $L'$ be a separable
cubic field extension of $K$. Let $\varphi\colon\Gamma_{K}\to S_{n}$
be the map corresponding to the \'etale $K$-algebra $L=L'\times K^{n-3}$.
If\, $L'/K$ is not Galois, then $\Image(\varphi)$ contains a transposition. 
\end{lem}

\begin{proof} 
Let us regard $S_{3}$ as the subgroup of $S_{n}$ which permutes
$1$, $2$ and $3$ in $\{1,2,\dots,n\}$. Then, after a suitable
rearrangement of elements of $\{1,\dots,n\}$, we get $\Image(\varphi)\subset S_{3}$.
If $L'/K$ is not Galois, then the Galois closure has Galois group
isomorphic to $S_{3}$. This means that $\Image(\varphi)$ is $S_{3}$.
In particular, it contains the transposition $(1,2)$. 
\end{proof}

\begin{lem}
\label{lem:dim-cyc-cubic}Suppose $p=3$. Let $\Delta_{3,d,\Gal}^{\circ}\subset\Delta_{3,d}^{\circ}$
be the locus of Galois extensions. Then
\[
\dim\Delta_{3,d,\Gal}^{\circ}-\frac{d}{2}\le-1.
\]
\end{lem}

\begin{proof}
For $d>0$, let $\Delta_{\ZZ/3\ZZ,d}$ denote the locus of $\ZZ/3\ZZ$-torsors
with discriminant exponent $d$. The locus $\Delta_{3,d,\Gal}^{\circ}$
is the image of the forgetting map $\Delta_{\ZZ/3\ZZ,d}\to\Delta_{3}^{\circ}$.
If $L/K$ is a cubic Galois extension with ramification jump $j$
($j>0$, $p\nmid j$), then $d=2(j+1)$; see \cite[Section~V.3, Lemma 3]{serre1979localfields}.
Thus, the locus $\Delta_{\ZZ/3\ZZ,d}$ is also the locus with ramification
jump $j={d}/{2}-1$. From \cite[Proposition 2.11]{yasuda2014thepcyclic},
we have
\[
\Delta_{\ZZ/3\ZZ,d}=\GG_{m}\times\AA_{k}^{j-\lfloor j/3\rfloor-1}.
\]
In particular, $\dim\Delta_{\ZZ/3\ZZ,d}=j-\lfloor j/3\rfloor$. Thus,
\begin{align*}\pushQED{\qed}
\dim\Delta_{3,d,\Gal}^{\circ}-\frac{d}{2} & =j-\left\lfloor \frac{j}{3}\right\rfloor -\frac{2(j+1)}{2}\\
 & =-\left\lfloor \frac{j}{3}\right\rfloor -1\\
 & \le-1.\qedhere \popQED
\end{align*}
\renewcommand{\qed}{}    
\end{proof}

\begin{prop}
\label{prop:key-prop}Suppose that $G$ contains no transposition.
Then, for any positive integer $d$,
\[
\dim\Delta_{G,d}-\frac{d}{2}\le-1.
\]
\end{prop}

\begin{proof}
Let $\nu=(\nu_{1},\dots,\nu_{l})$ be a partition of $n$ by positive
integers satisfying $\nu_{1}\ge\nu_{2}\ge\cdots\ge\nu_{l}$, and let
$\delta=(\delta_{1},\dots,\delta_{l})$ be a partition of $d$ by
non-negative integers such that 
\begin{gather*}
\delta_{i}>0\Longleftrightarrow\nu_{i}>1.
\end{gather*}
We write $\Delta_{\nu,\delta}:=\Image(\eta_{\nu,\delta})$. Then,
$\Delta_{n}=\bigcup_{(\nu,\delta)}\Delta_{\nu,\delta},$ where $(\nu,\delta)$
runs over pairs of partitions as above. Moreover, if $o\in\Delta_{n}$
denotes the point corresponding to the trivial cover
\[
\overset{n\text{ copies}}{\overbrace{\Spec K\amalg\cdots\amalg\Spec K}}\lra\Spec K,
\]
then $\{o\}=\Delta_{(1,\dots,1),(0,\dots,0)}$ and 
\[
\Delta_{n}\setminus\{o\}=\bigcup_{(\nu,\delta)\ne((1,\dots,1),(0,\dots,0))}\Delta_{\nu,\delta}.
\]
We define $\Delta_{\nu,\delta}^{(G)}:=\Image(\psi_{G})\cap\Delta_{\nu,\delta}$,
where $\psi_{G}$ is the map $\Delta_{G}\to\Delta_{n}$ introduced
at the beginning of Section~\ref{subsec:Dimensions-of-loci}. To prove
the proposition, it suffices to show that for every pair $(\nu,\delta)$
as above other than $((1,\dots,1),(0,\dots,0))$, we have
\begin{equation}
\dim\Delta_{\nu,\delta}^{(G)}-\frac{\sum_{i}\delta_{i}}{2}\le-1.\label{eq:ineq-wanted}
\end{equation}
Note that 
\begin{align*}
 & \dim\Delta_{\nu,\delta}^{(G)}-\frac{\sum_{i}\delta_{i}}{2}\\
 & \le\dim\Delta_{\nu,\delta}-\frac{\sum_{i}\delta_{i}}{2}.\\
 & =\sum_{i}\left(\dim\Delta_{\nu_{i},\delta_{i}}^{\circ}-\frac{\delta_{i}}{2}\right).
\end{align*}

{\it Case} $\nu_{1}\ge4$: From Corollary~\ref{cor:loci-dim-bound}, Inequality
(\ref{eq:ineq-wanted}) holds in this case. 

{\it Case} $\nu_{1}=3$ {\it and} $p\ne3$: Again, from Corollary~\ref{cor:loci-dim-bound},
Inequality (\ref{eq:ineq-wanted}) holds.

{\it Case} $\nu_{1}=3$ {\it and} $p=3$: If $\nu$ has at least two entries and
$\nu_{2}\ge2$, then 
\begin{align*}
\sum_{i}\left(\dim\Delta_{\nu_{i},\delta_{i}}^{\circ}-\frac{\delta_{i}}{2}\right) & \le\left(\dim\Delta_{\nu_{1},\delta_{1}}^{\circ}-\frac{\delta_{1}}{2}\right)+\left(\dim\Delta_{\nu_{2},\delta_{2}}^{\circ}-\frac{\delta_{2}}{2}\right)\\
 & \le-\frac{1}{2}-\frac{1}{2}\\
 & =-1.
\end{align*}
If $\nu$ is of the form $(3,1,\dots,1)$ (including the case $\nu=(3)$),
then from Lemma~\ref{lem:non-Gal-trans}, we have 
\[
\Delta_{\nu,\delta}^{(G)}\subset\eta_{\nu}\left(\Delta_{3,d,\Gal}^{\circ}\times\left(\Delta_{1,0}^{\circ}\right)^{d-3}\right).
\]
Since $\dim(\Delta_{1,0}^{\circ})^{d-3}=0$, from Lemma~\ref{lem:dim-cyc-cubic},
\[
\dim\Delta_{\nu,\delta}^{(G)}-\frac{d}{2}\le\dim\Delta_{3,d,\Gal}^{\circ}-\frac{d}{2}\le-1.
\]

{\it Case} $\nu_{1}=2$ {\it and} $p\ne2$: For a quadratic extension $L'/K$,
the map $\Gamma_{K}\to S_{n}$ corresponding to $L=L'\times K^{d-2}$
has image $\langle(1,2)\rangle\subset S_{n}$, which contains the
transposition $(1,2)$. From the assumption that $G$ contains no
transposition, $\Delta_{\nu,\delta}^{(G)}$ is non-empty only when
$\nu$ has at least two entries and $\nu_{1}=\nu_{2}=2$. Then,
\begin{align*}
\sum_{i}\left(\dim\Delta_{\nu_{i},\delta_{i}}^{\circ}-\frac{d_{i}}{2}\right) & \le\left(\dim\Delta_{\nu_{1},\delta_{1}}^{\circ}-\frac{d_{1}}{2}\right)+\left(\dim\Delta_{\nu_{2},\delta_{2}}^{\circ}-\frac{\delta_{2}}{2}\right)\\
 & \le-\frac{1}{2}-\frac{1}{2}\\
 & =-1.
\end{align*}

{\it Case} $\nu_{1}=2$ {\it and} $p=2$: From Lemma~\ref{lem:indep-trans}, if
an \'etale algebra of the form 
\[
L_{1}\times\cdots\times L_{m}\times K^{n-2m}
\]
with $L_{1},\dots,L_{m}$ quadratic extensions of $K$ corresponds
to a point of $\Delta_{\nu,\delta}^{(G)}$, then the classes $[L_{1}],\dots,[L_{m}]$ in $K/\wp(K)$
are linearly dependent over $\FF_{2}$. Using this fact, we will give
an upper bound of the dimension of~$\Delta_{\nu,\delta}^{(G)}$. 

Let $\nu=(\nu_{1},\dots,\nu_{l})$ be a partition of $n$ such that
$\nu_{1}=\dots=\nu_{r}=2$ and $\nu_{r+1}=\cdots=\nu_{l}=1$, and let
$\delta=(\delta_{1},\dots,\delta_{l})$ be a partition of $d$. For
$i\in\{1,\dots,r\}$ and $\ba=(a_{1},\dots,a_{r})\in\FF_{2}^{r}$
with $a_{i}=0$, we consider the map
\begin{align*}
\theta_{i,\ba}\colon\prod_{\substack{1\le j\le r\\
j\ne i
}
}\Delta_{2,\delta_{j}}^{\circ} & \lra\Delta_{n}\\
(L_{1},\dots,\check{L_{i}},\dots,L_{r}) & \longmapsto L_{1}\times\cdots\times\check{L_{i}}\times\cdots\times L_{r}\times\left(\sum_{\substack{1\le j\le r\\
j\ne i
}
}^ {}a_{j}L_{j}\right)\times K^{d-2r}.
\end{align*}
Here the symbol $\check{\cdot}$ means that the designated entry is
omitted, and $\sum_{j\ne i}a_{j}L_{j}$ means the quadratic extension
of $K$ corresponding to $\sum_{j\ne i}a_{j}[L_{j}]\in K/\wp K$.
The fact mentioned above implies that $\Delta_{\nu,\delta}^{(G)}$
is contained in $\bigcup_{i,\ba}\Image(\theta_{i,\ba})$. It follows
that if $\Delta_{\nu,\delta}^{(G)}\ne\emptyset$, then
\[
\dim\Delta_{\nu,\delta}^{(G)}\le\max_{\substack{i\\
\Delta_{\nu_{i},\delta_{i}}^{\circ}\ne\emptyset
}
}\sum_{\substack{1\le j\le r\\
j\ne i
}
}\dim\Delta_{\nu_{j},\delta_{j}}^{\circ}.
\]
Suppose that the maximum on the right side is attained at $i=i_{0}$.
Then, $\Delta_{\nu_{i_{0}},\delta_{i_{0}}}^{\circ}\ne\emptyset$, and
from Theorem~\ref{thm:dimensions-loci}, $\dim\Delta_{\nu_{i_{0}},\delta_{i_{0}}}^{\circ}>0$.
We conclude that

\begin{align*}\pushQED{\qed}
\dim\Delta_{\nu,\delta}^{(G)}-\frac{\sum_{i}\delta_{i}}{2} & \le\sum_{\substack{1\le j\le r\\
j\ne i_{0}
}
}\dim\Delta_{\nu_{j},\delta_{j}}^{\circ}-\frac{\sum_{i}\delta_{i}}{2}\\
 & \le\sum_{1\le i\le r}\left(\dim\Delta_{\nu_{i},\delta_{i}}^{\circ}-\frac{d}{2}\right)-\dim\Delta_{\nu_{i_{0}},\delta_{i_{0}}}^{\circ}\\
 & \le-\dim\Delta_{\nu_{i_{0}},\delta_{i_{0}}}^{\circ}\\
 & \le-1.\qedhere \popQED
\end{align*}
\renewcommand{\qed}{}    
\end{proof}

\section{Proofs of Theorems~\ref{thm:main-non-log} and~\ref{thm:main-log}\label{sec:Proof of main theorems}}

\subsection{Proof of Theorem~\ref{thm:main-log}}

We first prove Theorem~\ref{thm:main-log}. We follow the notation
from Section~\ref{sec:Permutation-actions}. From Corollary~\ref{cor:loci-dim-bound}(4) and Pro\-po\-si\-tion~\ref{prop:klt-lc}, the log pair $(X,B)$ is log
canonical. From Theorem~\ref{thm:dimensions-loci}, 
\begin{align*}
\dim\Delta_{\nu,\delta}^{(G)}-\frac{d}{2} & \le\sum_{i}\left(\dim\Delta_{\nu_{i},\delta_{i}}^{\circ}-\frac{\delta_{i}}{2}\right)\\
 & \le\sum_{i}\left(\left\lceil \frac{\delta_{i}-\nu_{i}+1}{p}\right\rceil -\frac{\delta_{i}}{2}\right).
\end{align*}
If $p>2$, since 
\[
\lim_{\delta_{i}}\left\lceil \frac{\delta_{i}-\nu_{i}+1}{p}\right\rceil -\frac{\delta_{i}}{2}=-\infty,
\]
we have 
\[
\lim_{\nu,\delta}\left(\dim\Delta_{\nu,\delta}^{(G)}-\frac{d}{2}\right)=-\infty.
\]
Proposition~\ref{prop:klt-lc} shows $(X,B)$ is Kawamata log terminal
if $p\ne2$. We have completed the proof of Theorem~\ref{thm:main-log}. 

\subsection{Proof of Theorem~\ref{thm:main-non-log} }

If $G$ contains no transpositions, then Theorem~\ref{thm:main-non-log}
follows from Propositions~\ref{prop:canonical-criterion} and~\ref{prop:key-prop}.
In the general case, we will prove Theorem~\ref{thm:main-non-log}
by reducing it to Theorem~\ref{thm:main-log} and the case without
transpositions.

\subsubsection*{Case $p\protect\ne2$}

Let $f\colon Y\to X$ be a proper birational morphism from a normal
variety $Y$, and let $E$ be a prime divisor on $Y$. If $f(E)\not\subset\Supp(B)$,
then $X$ has quotient singularities associated to permutation actions
without transpositions at general points of $f(E)$. Thus, from the
case without transposition, which was considered at the beginning
of this subsection, we have 
\[
\discrep{E;X,B}=\discrep{E;X}\ge0.
\]
If $f(E)\subset\Supp(B)$, then since $B$ is a 2-Cartier effective
divisor, $f^{*}B$ has multiplicity at least $1/2$ along $E$. From
Theorem~\ref{thm:main-log}, $(X,B)$ is Kawamata log terminal. Since
$K_{X}+B$ is 2-Cartier, $\discrep{E;X,B}\ge-1/2$. Thus, 
\begin{align*}
\discrep{E;X} & =\discrep{E;X,B}+\mult_{E}(f^{*}B).\\
 & \ge-1/2+1/2\\
 & \ge0.
\end{align*}
We have proved Theorem (\ref{thm:main-non-log}) when $p\ne2$. 

\subsubsection*{Case $p=2$}

Let $f\colon Y\to X$ be a proper birational morphism from a normal
variety $Y$, and let $E$ be a prime divisor on $Y$. If $f(E)\not\subset\Supp(B)$,
then the same argument as in the case $p\ne2$ shows that $\discrep{E;X}\ge0$.
If $f(E)\subset\Supp(B)$, then since $B$ is a Cartier effective
divisor, $f^{*}B$ has multiplicity at least $1$ along $E$. From
Theorem~\ref{thm:main-log}, $(X,B)$ is log canonical and hence $\discrep{E;X,B}\ge-1$.
Thus, 
\begin{align*}
\discrep{E;X} & =\discrep{E;X,B}+\mult_{E}(f^{*}B).\\
 & \ge-1+1\\
 & \ge0.
\end{align*}
We have proved Theorem (\ref{thm:main-non-log}) in the case $p=2$.

%%%%%%%%%%%%%%%%%%%%%
% References
%%%%%%%%%%%%%%%%%%%%%

\end{document}